\documentclass[reqno]{amsart}
\usepackage{diagrams}
\usepackage{amssymb}
\usepackage{mathrsfs}
\usepackage{cite}
\usepackage{tikz-cd}
\usepackage{MnSymbol}
\usepackage{a4wide}
 
\usepackage{xcolor}

\DeclareMathOperator\C{\mathbb C}

\DeclareMathOperator\Z{\mathbb Z}
\newcommand{\Om}{\Omega}

\newtheorem{theorem}{Theorem}[section]
\newtheorem{lemma}[theorem]{Lemma}

\newtheorem{prop}[theorem]{Proposition}
\theoremstyle{definition}

\theoremstyle{remark}

\newcommand{\dontprint}[1]\relax

\newcommand{\Ga}{\Gamma}

\newcommand{\A}{{\mathbb A}}

\newcommand{\ot}{\otimes}

\newcommand{\NN}{{\mathcal N}}

\newcommand{\OO}{{\mathcal O}}

\newcommand{\VV}{{\mathcal V}}

\newcommand{\XX}{{\mathcal X}}

\newcommand{\si}{\sigma}
\newcommand{\de}{\delta}
\newcommand{\sub}{\subset}

\newcommand{\Ber}{\operatorname{Ber}}

\newcommand{\lan}{\langle}
\newcommand{\ran}{\rangle}

\newcommand{\im}{\operatorname{im}}

\renewcommand{\a}{\alpha}
\renewcommand{\b}{\beta}

\newcommand{\id}{\operatorname{id}}

\renewcommand{\th}{\theta}

\newcommand{\hra}{\hookrightarrow}

\newcommand{\Si}{\Sigma}

\numberwithin{equation}{section}
\title{A bound on the Hodge filtration of the de Rham cohomology of supervarieties}
\author{Alexander Polishchuk}
\thanks{Partially supported by the NSF grant DMS-2001224, 
and within the framework of the HSE University Basic Research Program and by the Russian Academic Excellence Project `5-100'.}
\address{
    Department of Mathematics, 
    University of Oregon, 
    Eugene, OR 97403, USA; and National Research University Higher School of Economics, Moscow, Russia
  }
  \email{apolish@uoregon.edu}
\author{Dmitry Vaintrob}
\address{IHES, Le Bois-Marie, Route de Chartres 91440, Bures-sur-Yvette, France}
\email{mvaintrob@ihes.fr}
\begin{document}

\begin{abstract} We study the relation between the Hodge filtration of the de Rham cohomology of a proper smooth supervariety $X$ and
the usual Hodge filtration of the corresponding reduced variety $X_0$.
\end{abstract}

\maketitle

\section{Introduction}

Let $X$ be a smooth proper supervariety of dimension $n|m$ over $\C$.
It is known that the de Rham cohomology $H^\bullet_{dR}(X)$ of $X$ is canonically isomorphic to the de Rham cohomology $H^\bullet_{dR}(X_0)$ of the
underlying usual variety $X_0$ of dimension $n$. However, the Hodge filtration $(F^pH^i_{dR}(X))$
on $H^i_{dR}(X)$ in general differs from the Hodge filtration $(F^pH^i_{dR}(X_0))$ on $H^i_{dR}(X_0)$
(see \cite{P-dR}). It is natural to try to compare these Hodge filtrations. There is an obvious inclusion 
$$F^pH^i_{dR}(X)\sub F^pH^i_{dR}(X_0)$$
which we can think of as an upper bound for $F^pH^i_{dR}(X)$ in terms of the Hodge filtration of $X_0$.

Our main result gives a lower bound for $F^1H^i_{dR}(X)$ in terms of the Hodge filtration of $X_0$.
It can also be formulated in terms of the canonical map
\begin{equation}\label{kappa-eq}
\kappa_{n+i}:H^i(X,\Ber_X)\to H^{n+i}_{dR}(X_0),
\end{equation}
defined either using Serre duality or using the complex of integral forms (see Sec.\ \ref{int-sec}). 

\medskip

\noindent
{\bf Theorem A}. {\it For each $i$, one has an inlcusion $F^1H^i_{dR}(X)\supset F^{m+1}H^i_{dR}(X_0)$ (recall that the dimension of $X$ is $n|m$).
Equivalently, one has an inclusion
\begin{equation}\label{main-kappa-eq}
\im(\kappa_{n+i})\sub F^{n-m}H^{n+i}_{dR}(X_0).
\end{equation}
}

\medskip

The equivalence in Theorem A is a consequence of a certain duality statement: we prove that the $\kappa_{n+i}$ is dual to the natural
map $H^{n-i}_{dR}(X)\to H^{n-i}(X,\OO)$, with respect to the Serre duality and the Poincare duality (see Proposition \ref{p0-prop}).
The inclusion \eqref{main-kappa-eq} was conjectured by Kontsevich.

It is natural to ask how to generalize Theorem A to give lower bounds for $F^pH^i_{dR}(X)$ for $p>1$.
We prove a certain generalization of \eqref{main-kappa-eq}, which compares an analog of Hodge filtration for the hypercohomology
of the complex of integral forms on $X$ with the classical Hodge filtration (see Theorem \ref{int-thm}).
Conjecturally, this should lead to the inclusion
$$F^{p+1}(X)\supset F^{2p+m+1}(X_0).$$
(see Sec.\ \ref{int-sec}).

The crucial role in our proof is played by a certain natural family of supervarieties $\XX\to \A^1$, whose fiber over $t\neq 0$ is isomorphic to $X$ and whose
fiber over $0$ is the split supervariety $X_{sp}$ associated with $X$. This family is a superanalog of the standard deformation to the normal cone for the inclusion
$X_0\sub X$.

\bigskip

\noindent
{\it Acknowledgment}. The second author is grateful to Maxim Kontsevich for posing the problem and for useful discussions.

\bigskip

%\noindent
%{\it Conventions}. 
%For a $\Z/2$-graded object $X$ we denote by $X_+$ and $X_-$ its even and odd parts.
%We use the parity of differential forms such that the de Rham differential is even, so the de Rham complex can be viewed as a complex of $\Z/2$-graded sheaves.

\section{Hodge filtrations for the de Rham complex and for the complex of integral forms}\label{deR-sec}

\subsection{De Rham cohomology}

Recall that for a smooth proper supervariety $X$ we have the de Rham complex $\Om^\bullet_X$ (placed in degrees $\ge 0$). 
The de Rham cohomology $H^\bullet_{dR}(X)$ is defined as the hypercohomology of $\Om^\bullet$.

Let $i:X_0\hra X$ be the bosonization of $X$, so $\OO_{X_0}=\OO_X/\NN$, where $\NN$ is the ideal generated by odd functions.
The natural map $\Om^\bullet_X\to i_*\Om^\bullet_{X_0}$ is known to be a quasi-isomorphism (see e.g., \cite{P-dR}), so it induces isomorphisms
\begin{equation}\label{dR-X-X0-eq}
H^q_{dR}(X)\simeq H^q_{dR}(X_0)\simeq H^q(X_0,\C),
\end{equation}
where $H^q(X_0,\C)$ denotes the cohomology of the constant sheaf with respect to the analytic topology.
Equivalently, the inclusion of the constant sheaf $\C_X$ into $\Om^\bullet_X$ induces an isomorphism on cohomology (in the analytic topology).

%Relative versions, Gauss-Manin connection.

The Hodge filtration on de Rham cohomology is defined by 
$$F^pH^i_{dR}(X):=\im(H^i(X,\si_{\ge p}\Om^\bullet)\to H^i(X,\Om^\bullet),$$
where $\si_{\ge p}$ denotes the stupid truncation of a complex.
Equivalently, $F^pH^i_{dR}(X)$ is the kernel of the map to $H^i(X,\si_{\le p-1}\Om^\bullet)$.

The commutative diagram
\begin{equation}\label{Hodge-X-X0-diagram}
\begin{diagram}
H^i(X,\Om^{\ge p}_X)&\rTo{}& H^i(X_0,\Om^{\ge p}_{X_0})\\
\dTo{}&&\dTo{}\\
H^i(X,\Om^\bullet_X)&\rTo{}& H^i(X_0,\Om^\bullet_{X_0})
\end{diagram}
\end{equation}
shows that under the identification \eqref{dR-X-X0-eq}
we have an inclusion $F^pH^i_{dR}(X)\sub F^pH^i_{dR}(X_0)$.

Recall that $X$ is called {\it split} if there is an isomorphism of algebras
$\OO_X\simeq \bigwedge_{\OO_{X_0}}(\VV)$ for a vector bundle $\VV$ over
$X_0$. A weaker condition is that $X$ is {\it projected}, i.e., there exists a projection $p:X\to X_0$ such that the composition
$X_0\rTo{i} X\rTo{p} X_0$ is the identity map.

\begin{lemma}\label{proj-de-rham-lem}
Assume that $X$ is projected. Then under the identification \eqref{dR-X-X0-eq} one has $F^pH^i_{dR}(X)=F^pH^i_{dR}(X_0)$.
\end{lemma}

\begin{proof}
Let $\pi:X\to X_0$ be a projection. Then the pull-back by $\pi$ gives a map
$\pi^*:H^i(\Om^{\ge p}_{X_0})\to H^i(\Om^{\ge p}_X)$ whose composition with the pull-back by $i$ is the identity
on $H^i(\Om^{\ge p}_{X_0})$. Hence, the top horizontal map in the diagram \eqref{Hodge-X-X0-diagram}
is surjective, which implies the assertion.
\end{proof}

%Griffiths transversality.

%Compatibility with the base change in the case of even base???
%Let $\pi:\XX\to S$ be a smooth proper family, $u:S'\to S$ a base change, $\pi':\XX'\to S'$ the induced family.

\subsection{Complex of integral forms}\label{int-sec}

We denote by $\Ber_X$ the Berezinian line bundle on $X$ (obtained as the Berezinian of $\Om^1_X$).
Recall that the complex of integral forms $\Si_{\bullet}=\Si_{X,\bullet}$ is a complex of the form
$$\ldots \Ber_X\ot {\bigwedge}^2 T_X\rTo{\de} \Ber_X\ot T_X\rTo{\de} \Ber_X,$$
placed in degrees $\le n$ (so that $\Si_n=\Ber_X$). The formula for the differential $\de$ that uses the right connection on $\Ber_X$ can
be found in \cite[ch.\ IV, 5.4]{Manin}, \cite{Penkov}.

It is known (see \cite{Penkov}) that there is a natural isomorphism
\begin{equation}\label{int-forms-coh-eq}
H^i(X_0,\C)=H^i(X_0,\Om^\bullet)\rTo{\sim} H^i(\Si_\bullet)
\end{equation}
and that the map
\begin{equation}\label{Ber-comp-eq}
H^n(\Ber)\to H^{2n}(\Si_{\bullet})\simeq H^{2n}(X_0,\C)\simeq \C
\end{equation}
is an isomorphism.

Mimicking the de Rham case, we can define the Hodge filtration on hypercohomology integral forms. 
Namely, we define $F_{n-p}H^i(\Si_{\bullet})$ to be the image of the map
$$\kappa_i^{n-p}:H^i(\si_{\ge n-p}\Si_{\bullet})\to H^i(\Si_{\bullet}).$$
For example, the image of the map $\kappa_{n+i}=\kappa_{n+i}^n$ considered in \eqref{kappa-eq} is $F_nH^{n+i}(\Si_\bullet)$.

%Claim: the map 
%$$H^{n-i}(\si_{\ge n-p}\Si_{\bullet})\to H^{n-i}(\Si_{\bullet})$$
%is dual to the map 
%$$H^{n+i}(\Om^\bullet)\to H^{n+i}(\si_{\le p}\Om^\bullet).$$

The next result is a generalization of the inclusion \eqref{main-kappa-eq}. It will be proved in Section \ref{Def-sec}.

\begin{theorem}\label{int-thm}
One has an inclusion
$$F_{n-p}H^i(\Si_X)\sub F_{n-2p-m}H^i(X_0,\C).$$
\end{theorem}

It is natural to conjecture that the Hodge filtrations on the de Rham cohomology and on the hypercohomology of the complex of integral forms are related.

\medskip

\noindent
{\bf Conjecture B}.
{\it $F^{p+1}H^{i}(\Om^\bullet_X)$ is equal to the orthogonal of $F_{n-p}H^{2n-i}(\Si_{X,\bullet})$ with respect to the Poincare pairing on 
$H^\bullet(X_0,\C)$.
}

\medskip

Note that the analog of Conjecture B for usual (even) varieties is an easy corollary of the Hodge decomposition.
If Conjecture B is true then the statement of Theorem \ref{int-thm} is equivalent to the inclusion
$$F^{p+1}H^i_{dR}(X)\supset F^{2p+m+1}H^i_{dR}(X_0).$$

We can check Conjecture B for $p=0$. 

\begin{prop}\label{p0-prop} 
The map $\kappa_{n+i}:H^i(X,\Ber_X)\to H^{n+i}(X_0,\C)$ is dual to the natural map $p_{n-i}:H^{n-i}(X_0,\C)\to H^{n-i}(X,\OO_X)$,
where we use the Poincare pairing to identify $H^{n-i}(X_0,\C)$ with the dual of $H^{n+i}(X_0,\C)$, and we use the Serre duality to
identify $H^{n-i}(X,\OO_X)$ with the dual of $H^i(X,\Ber_X)$.
Hence, Conjecture B holds for $p=0$.
\end{prop}

\begin{proof}
For a pair of classes $\a\in H^i(X,\Ber_X)$ and $\b\in H^{n-i}(X_0,\C)$, we have to check the equality
$$\lan \kappa_i(\a),\b\ran_{PD}=\lan \a, p_{n-i}(\b)\ran_{SD},$$
where on the left we use the Poincare duality and on the right we use the Serre duality.

By definition, the Serre duality is induced by the cup product and by the canonical functional on $H^n(X,\Ber_X)$.
The latter canonical functional can be computed as the composition \eqref{Ber-comp-eq}.
Hence, we need to check the following identity in $H^{2n}(X_0,\C)$:
$$\kappa_i(\a)\cup \b=\kappa_n(\a\cup p_{n-i}(\b))=\kappa_n(\a\cup \b),$$
where the cup-product $\a\cup \b$ is induced by the natural pairing of sheaves $\Ber_X\ot \C_X\to \Ber_X$.
Now the assertion follows from the fact that the cup product on de Rham cohomology is induced by the sheaf pairing
$\Si_{X,\bullet} \ot \C_X\to \Si_{X,\bullet}$.
\end{proof}

Another case when Conjecture B is true is the case of a projected supervariety $X$, as follows from Lemma \ref{proj-de-rham-lem} and
the next result. 

\begin{lemma}\label{proj-int-lem} 
Assume that $X$ is projected. Then under the identification \eqref{int-forms-coh-eq} we have
$$F_{n-p}H^i(\Si_X)=F^{n-p}H^i_{dR}(X_0).$$
%equipped with a splitting, i.e., an isomorphism of algebras $\OO_X\simeq \bigwedge_{\OO_{X_0}}(\VV)$ for a vector bundle $\VV$ over
%$X_0$. Then it induces an isomorphism
%$$H^i(\si_{\ge n-p}\Si_{X,\bullet})\simeq H^i(\si_{\ge n-p}\Om_{X_0}),$$
%for each $i,p$, such that the map $\kappa_i^{n-p}$ gets identified with the standard map
%$$H^i(\si_{\ge n-p}\Om_{X_0})\to H^i(\Om_{X_0}).$$
\end{lemma}

\begin{proof}
Let $\pi:X\to X_0$ be a projection. It induces projections from integral forms on $X$ to those on $X_0$, which gives a commutative
diagram
\begin{diagram}
H^q(X,\si_{\ge n-p}\Si_X)&\rTo{\pi_*}&H^q(X_0,\si_{\ge n-p}\Si_{X_0})\\
\dTo{}&&\dTo{}\\
H^q(X,\Si_X)&\rTo{\pi_*}&H^q(X_0,\si_{\ge n-p}\Si_{X_0})
\end{diagram}
Here the bottom horizontal arrow is the inverse of the isomorphism \eqref{int-forms-coh-eq}.
On the other hand, the inlcusion $i:X_0\hra X$ induces the map
$$i_*:H^q(X_0,\si_{\ge n-p}\Si_{X_0})\to H^q(X,\si_{\ge n-p}\Si_X)$$
such that $\pi_*\circ i_*=\id$. This shows that the top horizontal arrow in the above diagram is surjective, and
the assertion follows.
\end{proof}

\subsection{Gauss-Manin connection}

There is a relative version $\Si_{\XX/S}$ of the complex of integral forms for a smooth proper morphism $\pi:\XX\to S$. 
Furthermore, for smooth $S$, there is a natural Gauss-Manin connection on $R\pi_*\Si_{\XX/S}$. 
%For simplicity we consider
%only the case when $S$ is a smooth of dimension $1|0$ (we will only use this case in the paper).

The distribution $T_{\XX/S}\sub T_{\XX}$ induces a filtration $G_0\Si_{\bullet}\sub G_1\Si_{\XX,\bullet}\sub\ldots \Si_{\XX,\bullet}$ by subcomplexes in the complex of integral forms on $\XX$, where
$$G_0\Sigma_{\XX,n-p}:=\Ber_{\XX}\ot {\bigwedge}^p(T_{\XX/S}), \ \ G_1\Sigma_{\XX,n-p}:=\Ber_{\XX}\ot {\bigwedge}^{p-1}(T_{\XX/S})\cdot T_\XX, \ldots$$
Using the identification $\Ber_{\XX}\simeq \pi^*\Ber_S\ot \Ber_{\XX/S}$, we get isomorphisms of complexes
$$G_0\Sigma_\XX\simeq \pi^*\Ber_S\ot \Sigma_{\XX/S}, \ \ G_1\Sigma_\XX/G_0\Sigma_\XX\simeq \pi^*\Ber_S\ot \Sigma_{\XX/S}[1]\ot \pi^*T_S,$$
Thus, a short exact sequence of complexes
\begin{equation}\label{G0-G1-seq}
0\to G_0\to G_1\to G_1/G_0\to 0
\end{equation}
leads to an exact triangle
$$\Ber_S\ot R\pi_*\Sigma_{\XX/S}\to R\pi_*G_1\to \Ber_S\ot R\pi_*\Sigma_{\XX/S}[1]\ot \pi^*T_S\to \Ber_S\ot R\pi_*\Sigma_{\XX/S}[1]$$
We can view the corresponding morphism
$$\nabla_r: \Ber_S\ot R\pi_*\Sigma_{\XX/S}\ot \pi^*T_S\to \Ber_S\ot R\pi_*\Sigma_{\XX/S}$$
as a right connection on $\Ber_S\ot R\pi_*\Sigma_{\XX/S}$, inducing right connections on the cohomology sheaves. 
Using the next term of the filtration, $G_2$, one can check its flatness (but we don't use it in this paper, since we only consider the case $S=\A^1$).

By the standard relation between left and right connections (see \cite{Penkov}), we get a left connection
$$\nabla: R\pi_*\Sigma_{\XX/S}\to \pi^*\Om^1_S\ot R\pi_*\Sigma_{\XX/S}.$$ 

Considering stupid truncations of the exact sequence \eqref{G0-G1-seq}, we get an exact triangle
$$\Ber_S\ot R\pi_*\si_{\ge n-p-1}\Sigma_{\XX/S}\to R\pi_*\si_{\ge n-p-1}(G_1)\to \Ber_S\ot R\pi_*(\si_{\ge n-p}\Sigma_{\XX/S})[1]\ot \pi^*T_S\to 
\Ber_S\ot R\pi_*\si_{\ge n-p-1}\Sigma_{\XX/S}$$
we get a commutative diagram
\begin{diagram}
R\pi_*(\si_{\ge n-p}\Sigma_{\XX/S})&\rTo{\nabla}& \pi^*\Om^1_S\ot R\pi_*(\si_{\ge n-p-1}\Sigma_{\XX/S})\\
\dTo{}&&\dTo{}\\
R\pi_*\Sigma_{\XX/S}&\rTo{\nabla}&\pi^*\Om^1_S\ot R\pi_*\Sigma_{\XX/S}
\end{diagram}
which gives an analog of Griffiths transversality for $\nabla$.

%Base change??? (at least in the case of even base???).

%Let $\pi:\XX\to S$ be a smooth proper family, $u:S'\to S$ a base change, $\pi':\XX'\to S'$ the induced family.
%For every $p$ we have a natural morphism
%$$Lu^*R\pi_*(\si_{\ge n-p}\Si_{\XX/S})\to R\pi'_*(\si_{\ge n-p}\Si_{\XX'/S'})$$

\section{Deformation to normal cone}\label{Def-sec}

\subsection{Construction}

We consider an analog of the deformation to normal cone for the embedding $X_0\hra X$.
This will be a family of supervarieties $\pi:\XX\to \A^1$, equipped with an identification 
\begin{equation}\label{XX-nonzero-eq}
\pi^{-1}(\A^1\setminus\{0\})=X\times (\A^1\setminus\{0\}),
\end{equation}
and an identification of $\pi^{-1}(0)$ with the split supervariety $X_{sp}$ associated with $X$.
In addition, we will have a natural affine morphism $f:\XX\to X$, restricting to the natural projection on $\pi^{-1}(\A^1\setminus\{0\})$.

We define $\XX$ using the following $\Z$-graded quasicoherent sheaf of $\OO_X$-algebras:
$$f_*\OO_{\XX}=\ldots\oplus \NN^2\cdot t^{-2}\oplus \NN\cdot t^{-1}\oplus \OO_X\oplus \OO_X\cdot t\oplus \OO_X\cdot t^2\oplus \ldots,$$
where $t$ is a formal variable, with the obvious algebra structure.

The map to $X\times \A^1$ corresponds to the natural inclusion $\OO_X[t]\hra f_*\OO_{\XX}$.
On the other hand, the embedding $f_*\OO_{\XX}\hra \OO_X[t,t^{-1}]$ corresponds to inverting $t$, so it gives an identification
\eqref{XX-nonzero-eq}.

Finally, the fiber over $t=0$ corresponds to the algebra
$$f_*\OO_{\XX}/t\cdot f_*\OO_{\XX}\simeq \ldots\oplus \NN^2/\NN^3\oplus \NN/\NN^2\oplus \OO_X/\NN,$$
which gives the split supervariety $X_{sp}$.

Note also that the bosonization of $\XX$ is canonically identified with $X_0\times \A^1$.

\begin{lemma}\label{power-t-lem}
(i) One has an equality
$$t^m f^*\Ber_X=\Ber_{\XX/\A^1},$$
where $(n|m)$ is the dimension of $X$ (we view both sides as subsheaves in the sheaf of rational sections of $f^*\Ber_X$).

\noindent
(ii) For $p\ge 0$, one has an inclusion
$$t^{m+p} f^*\Si_{X,n-p}\sub \Si_{\XX/\A^1,n-p}.$$
\end{lemma}

\begin{proof}
(i) The statement is clear over $\A^1\setminus\{0\}$, so it is enough to check it in the neighborhood of each point of $\pi^{-1}(0)$. 
Let $(x_1,\ldots,x_n,\th_1,\ldots,\th_m)$ be local coordinates of the corresponding point of $X$. Then
$(x_1,\ldots,x_n,\th_1/t,\ldots,\th_m/t)$ are relative local coordinates for $\XX/\A^1$, which immediately implies the assertion.

\noindent
(ii) This follows from (i) and from the inclusion $t\cdot f^*T_X\sub T_{\XX/S}$ which is proved using the same relative local coordinates as in (i).
\end{proof}

\begin{lemma}\label{integral-forms-coh-triv-lem}
For each $i$, the canonical identification 
$$H^i R\pi_*\Si_{\XX/\A^1}|_{\A^1\setminus 0}\simeq H^i(X,\Si_X)\ot \OO_{\A^1\setminus 0}\simeq H^i(X_0,\C)\ot \OO_{\A^1\setminus 0}$$
extends to an identification
$$H^i R\pi_*\Si_{\XX/\A^1} \simeq H^i(X_0,\C)\ot \OO_{\A^1}.$$
\end{lemma}

\begin{proof}
Let $b:X_0\times \A^1\to \XX$ denote the embedding of the bosonization. Then
the assertion follows from the fact that the natural map
$$b_*\Om^\bullet_{X_0\times \A^1/\A^1}\to \Si_{\XX/\A^1}$$
induces an isomorphism on $R\pi_*$.
\end{proof}

\subsection{Proofs of Theorem \ref{int-thm} and of Theorem A}

\begin{proof}[Proof of Theorem \ref{int-thm}]
Let us denote by $\kappa^{n-p}_{\XX}:\si_{\ge n-p}\Si_{\XX/\A^1}\to \Si_{\XX/\A^1}$ the natural inclusion.

Below we denote by $\nabla_t$ the Gauss-Manin connection with respect to $d/dt$.
We will also not make a distinction between quasicoherent sheaves on $\A^1$ and the corresponding modules over $\C[t]$
given by global sections.

From Lemma \ref{power-t-lem}(ii) we get an inclusion of complexes 
$$\tau_{m+p}:f^*\si_{\ge n-p}\Si_X\to \si_{\ge n-p}\Si_{\XX/\A^1},$$
which restricts to $\cdot t^{m+p}$ over $\A^1\setminus \{0\}$.

Thus, starting with a class $\a\in H^q(X,\si_{\ge n-p}\Si_X)$, we get the class
$$\tau_{m+p}f^*\a\in H^q(\XX,\si_{\ge n-p}\Si_{\XX/\A^1})\simeq H^qR\pi_*\si_{\ge n-p}\Si_{\XX/\A^1}$$
which restricts to $t^{m+p}\cdot \a$ over $\A^1\setminus \{0\}$.

We claim that 
$$\kappa^{n-p}_{\XX}(\tau_{m+p}f^*\a)=t^{m+p}\cdot \kappa^{n-p}(\a),$$
where we use the identification 
\begin{equation}\label{HqX[t]-eq}
H^qR\pi_*\Si_{\XX/\A^1}\simeq H^q(X,\Si_X)[t]
\end{equation}
(see Lemma \ref{integral-forms-coh-triv-lem}).
Indeed, it is enough to check this identity after inverting $t$,
in which case it follows from the triviality of the family over $\A^1\setminus\{0\}$ (see \eqref{XX-nonzero-eq}).

Similarly, we observe that under the identification \eqref{HqX[t]-eq}, the Gauss-Manin connection $\nabla_t$ on the left corresponds
to the usual differentiation $d/dt$ on the right (indeed, this is clear after inverting $t$).
Hence, using the commutative diagram on $\A^1$,
\begin{diagram}
H^qR\pi_*(\si_{\ge n-p}\Si_{\XX/\A^1})&\rTo{\nabla}& H^qR\pi_*(\si_{\ge n-p-1}\Si_{\XX/\A^1})\\
\dTo{\kappa^{n-p}_{\XX}}&&\dTo{\kappa^{n-p-1}_{\XX}}\\
H^qR\pi_*\Si_{\XX/\A^1}&\rTo{\nabla}& H^qR\pi_*\Si_{\XX/\A^1},
\end{diagram}
we deduce that
\begin{equation}\label{kappa-(n-2p-m)-eq}
\kappa^{n-2p-m} \bigl(\nabla^{m+p}\tau_{m+p}f^*\a\bigr)=(m+p)!\cdot \kappa^{n-p}(\a).
\end{equation}

Let $\iota:\{0\}\hra \A^1$ denote the embedding of the origin. For a complex of flat $\pi^{-1}\OO_{\A^1}$-modules $F$ on $\XX$,
we have a natural base change morphism 
%Recall base change??? It gives isomoprhisms in derived category,
$$L\iota^*R\pi_*F\to R\Ga(\pi^{-1}(0),j^\bullet F\otimes_{\C[t]}\C),$$
where $j:\pi^{-1}(0)\hra \XX$ is the natural embedding.
On the other hand, since $\iota_*$ is exact, for each $q\in \Z$, we have a morphism
$$H^q R\pi_*F\to H^q \iota_*L\iota^*R\pi_*F\simeq \iota_*H^q L\iota^*R\pi_*F.$$
Combining it with the base change morphism, we get natural maps on $\A^1$,
$$H^q R\pi_*F\to \iota_*H^q(\pi^{-1}(0),j^\bullet F\otimes_{\C[t]}\C).$$
%$$Li_0^*R\pi_*\si_{\ge n-p}\Si_{\XX/\A^1}\rTo{\sim} R\Ga \si_{\ge n-p} \si_{\ge n-p}\Si_{X_{sp}}.$$

We can apply the above construction to $F=\Si_{\XX/\A^1}$ and to $F=\si_{\ge n-2p-m}\Si_{\XX/\A^1}$.
Taking into account the identification of $\pi^{-1}(0)$ with $X_{sp}$, we get a  commutative diagram of $\OO_{\A^1}$-modules
\begin{equation}
\label{res-0-diagram}
\begin{diagram}
H^qR\pi_*(\si_{\ge n-2p-m}\Si_{\XX/\A^1})&\rTo{}& \iota_*H^q(X_{sp},\si_{\ge n-2p-m}\Si_{X_{sp}})\\
\dTo{\kappa^{n-2p-m}}&&\dTo{}\\
H^qR\pi_*\Si_{\XX/\A^1}&\rTo{r}& \iota_*H^q(X_{sp},\Si_{X_{sp}})
\end{diagram}
\end{equation}
We start with the class $\nabla^{m+p}\tau_{m+p}f^*\a$ in the top left corner and look at the two resulting
elements in the bottom right corner.
By Lemma \ref{integral-forms-coh-triv-lem}, the map $r$ is identified with map of the evaluation at $t=0$, 
$$H^q(X_0,\C)[t]\to H^q(X_0,\C).$$
Hence, by \eqref{kappa-(n-2p-m)-eq},
the class $r\kappa^{n-2p-m} \bigl(\nabla^{m+p}\tau_{m+p}f^*\a\bigr)$ is equal to $(m+p)!\kappa^{n-p}(\a)$.
Finally, by Lemma \ref{proj-int-lem}, the image of the right vertical map in the diagram \eqref{res-0-diagram}
is contained in the Hodge subspace $F^{n-2p-m}H^q_{dR}(X_0)$.
%with the map $H^q(X_0,\si_{\ge n-2p-m}\Om_{X_0}^\bullet)\to H^q(X_0,\C)$.
Thus, we conclude that $(m+p)!\kappa^{n-p}(\a)$ lies in that subspace.
\end{proof}

\medskip

\begin{proof}[Proof of Theorem A]
The case $p=0$ of Theorem \ref{int-thm} gives the inclusion
$$\im(\kappa_{n+i}^n)=F_nH^{n+i}(\Si_X)\sub F_{n-m}H^{n+i}(X_0,\C).$$
By Proposition \ref{p0-prop}, passing to the orthogonals, we get the inclusion 
$$F^{m+1}H^{n-i}(X_0,\C)\sub F^1H^{n-i}_{dR}(X).$$
\end{proof}

\end{document}